\documentclass[11pt,a4paper]{article}

\usepackage[T1]{fontenc}
\usepackage[utf8]{inputenc}
\usepackage{lmodern}
\usepackage{microtype}
\usepackage{amsmath,amssymb,amsthm,mathtools}
\usepackage[left=1.02in,right=1.02in,top=0.86in,bottom=0.92in]{geometry}
\usepackage[hidelinks]{hyperref}
\usepackage{authblk}

\hypersetup{
  pdftitle={Counterexamples to Dujella's conjecture on integral points on the elliptic curve attached to a Diophantine triple},
  pdfauthor={Ana Jurasic and Matej Jurasic}
}

\newtheorem{theorem}{Theorem}[section]
\newtheorem{lemma}[theorem]{Lemma}
\newtheorem{proposition}[theorem]{Proposition}
\newtheorem{corollary}[theorem]{Corollary}
\theoremstyle{remark}
\newtheorem{remark}[theorem]{Remark}

\newcommand{\Z}{\mathbb Z}
\newcommand{\Q}{\mathbb Q}
\newcommand{\legendre}[2]{\left(\frac{#1}{#2}\right)}

\allowdisplaybreaks
\title{Counterexamples to Dujella's conjecture on integral points on the elliptic curve attached to a Diophantine triple}
\author[1]{Ana Jurasi\'{c}}
\author[2]{Matej Jurasi\'{c}}
\affil[1]{Faculty of Mathematics, University of Rijeka,
Radmile Matej\v{c}i\'{c} 2, HR-51000 Rijeka, Croatia\\
\texttt{ajurasic@math.uniri.hr}}
\affil[2]{University of Zagreb, Faculty of Electrical Engineering and Computing,
Unska 3, HR-10000 Zagreb, Croatia\\
\texttt{matej.jurasic@fer.hr}}
\date{}

\begin{document}

\maketitle
\vspace{-1.5em}

\begin{abstract}
A set $\{a,b,c\}$ of distinct positive integers is called a Diophantine triple if
$ab+1$, $ac+1$, and $bc+1$ are perfect squares. Dujella formulated the conjecture that the only integral $x$-coordinates on
the attached elliptic curve $$y^2=(ax+1)(bx+1)(cx+1)$$ are $0,d_-,d_+$, together
with $-1$ when the triple contains $1$, where $d_{\pm}=a+b+c+2abc
 \pm2\sqrt{(ab+1)(ac+1)(bc+1)}$.  The weaker question was stated as
Problem~4.8 in Dujella's list of open problems: must every integral point
with $x\ne-1$ make $ax+1$, $bx+1$, and $cx+1$ all perfect squares?  We construct infinitely many triples
$\{5,115,c\}$ admitting an integral point with $x\ne-1$ for which all
three factors are nonsquares.
\end{abstract}

\noindent\textbf{Keywords.} Diophantine triples; elliptic curves; integral points on elliptic curves;
generalized Pell equations; quadratic nonresidues.

\noindent\textbf{2020 Mathematics Subject Classification.} 11D09, 11G05.

\section{Introduction}

A set $\{a_1,\ldots,a_m\}$ of distinct positive integers is called a
\emph{Diophantine $m$-tuple} if $a_i a_j+1$ is a perfect square whenever
$i\ne j$.  The study of such sets has a long history, and their extension
problems are closely connected with elliptic curves (see, for example,
\cite{DujellaBook,Dujella2001}). If $\{a,b,c\}$ is a Diophantine triple, then 
$a,b,c$ are distinct positive integers, so the cubic polynomial $f(X)=(aX+1)(bX+1)(cX+1)$ has the three distinct rational roots. Hence, the curve
\begin{equation}\label{eq:curve}
 E_{a,b,c}:\qquad y^2=(ax+1)(bx+1)(cx+1)
\end{equation}is a nonsingular cubic with a rational point, for instance $(0,1)$, and therefore an elliptic curve over $\Q$. Every positive integer $d$ for which $\{a,b,c,d\}$ is a
Diophantine quadruple yields an integral point on \eqref{eq:curve} with
$x=d$ because each of $ad+1$, $bd+1$, and $cd+1$ is then a perfect square.
Conversely, if a positive integer $x\notin\{a,b,c\}$ makes the three factors in
\eqref{eq:curve} perfect squares, then $\{a,b,c,x\}$ is a Diophantine
quadruple. Conjecture~1.4.5
of Dujella's monograph \cite{DujellaBook} predicts that the only
integral $x$-coordinates on the elliptic curve \eqref{eq:curve} are $0$ and the two regular extension values
\begin{equation}\label{d+-}
 d_{\pm}=a+b+c+2abc
 \pm2\sqrt{(ab+1)(ac+1)(bc+1)},
\end{equation}
together with $x=-1$ when $1\in\{a,b,c\}$.  The weaker question was posed as
Problem~4.8 in Dujella's list of open problems \cite{DujellaOpen}: must every integral point on the elliptic curve \eqref{eq:curve}
with $x\ne-1$ make $ax+1$, $bx+1$, and $cx+1$ all perfect squares?  Partial positive answers
were known for several special families, including triples of the form
$\{1,3,c\}$ in \cite{DujellaPetho} and Fibonacci triples
$\{F_{2n},F_{2n+2},F_{2n+4}\}$ in \cite{Dujella2001}.

The present construction gives a negative answer to the weaker question. As a consequence, the stronger question is also answered in the negative. The construction rests on two main ideas.  First, if we take
\[
 c=abx+a+b,
\]
then we have $cx+1=(ax+1)(bx+1)$. Hence, the product of all three factors in \eqref{eq:curve} is automatically a square, without any
of the individual factors being required to be a square. Second, for the Diophantine pair $\{5,115\}$, the conditions that $5c+1$ and $115c+1$ have to
be squares reduce to the generalized Pell equation
\begin{equation}\label{eq:pell}
 T^2-23S^2=-22.
\end{equation}
Choosing exponents in suitable residue classes along one Pell orbit simultaneously guarantees the required divisibility condition on $c$ and the nonsquareness of all three factors.

\section{The factorization criterion and the Pell equation}

We begin with the elementary identity responsible for the exceptional
integral points on the elliptic curve \eqref{eq:curve}.

\begin{proposition}\label{prop:factorization}
Let $a,b,c$ be distinct positive integers such that
\begin{equation}\label{eq:compatibility}
 c\equiv a+b\pmod{ab}.
\end{equation}
Set $x=\frac{c-a-b}{ab}.$ Then, $cx+1=(ax+1)(bx+1)$
and consequently
\begin{equation}\label{eq:square-product}
 (ax+1)(bx+1)(cx+1)=(cx+1)^2.
\end{equation}
If $\{a,b,c\}$ is a Diophantine triple, then
$\bigl(x,\pm(cx+1)\bigr)$ is an integral point on $E_{a,b,c}$.
\end{proposition}

\begin{proof}
Condition \eqref{eq:compatibility} makes $x$ integral and gives
$c=abx+a+b$.  Hence,
\[
 (ax+1)(bx+1)=abx^2+(a+b)x+1=cx+1.
\]
Multiplication by $cx+1$ proves \eqref{eq:square-product}. Since $a,b,c$ are distinct and nonzero, the polynomial $f(X)=(aX+1)(bX+1)(cX+1)$ has three distinct roots. Hence $Y^2=f(X)$
defines a nonsingular cubic curve, and thus an elliptic curve over $\Q$. The identity \eqref{eq:square-product} shows that
$\bigl(x,\pm(cx+1)\bigr)$ is an integral point on this curve.\end{proof}

\begin{remark}\label{rem:criterion}If $x\ne-1$ and at least one of the three factors in
\eqref{eq:square-product} is a nonsquare, the point $\bigl(x,\pm(cx+1)\bigr)$  is a counterexample to
the assertion in Problem~4.8.\end{remark}

We apply Proposition~\ref{prop:factorization} to
\[
 a=5,\qquad b=115,
\]
which form a Diophantine pair because $5\cdot115+1=24^2$.  A positive
integer $c$, distinct from $5$ and $115$, is a Diophantine extension of this pair if
there are positive integers $S,T$ such that
\begin{equation}\label{eq:extension-squares}
 5c+1=S^2,\qquad 115c+1=T^2.
\end{equation}
Eliminating $c$ gives the equation \eqref{eq:pell}.

We now describe all positive solutions of the equation \eqref{eq:pell}.
\begin{lemma}\label{lem:pell-solutions}
Let
\[
    u=24+5\sqrt{23}.
\]
Every pair $(T,S)$ of positive integers satisfying the equation \eqref{eq:pell} has a unique representation in exactly one of the two forms
\begin{equation}\label{eq:first-pell-orbit}
    T+S\sqrt{23}=(1+\sqrt{23})u^k,\qquad k\geq0,
\end{equation}
or
\begin{equation}\label{eq:second-pell-orbit}
    T+S\sqrt{23}=(91+19\sqrt{23})u^k,\qquad k\geq0.
\end{equation}
\end{lemma}

\begin{proof} We apply Nagell's results for generalized Pell equations (see \cite[Theorems~108a and 109]{nagel}). The continued fraction
$\sqrt{23}=[4;\overline{1,3,1,8}]$ shows that the fundamental positive solution of \[ X^2-23Y^2=1 \] is $(X,Y)=(24,5)$. Hence, $u=24+5\sqrt{23}$ is the fundamental unit of norm $1$ in $\Z[\sqrt{23}]$.

Every class of solutions of the equation \eqref{eq:pell} contains a fundamental representative $T_0+S_0\sqrt{23}$ satisfying Nagell's bounds $0<S_0< \frac{5\sqrt{22}}{\sqrt{2(24-1)}} <4$ and $0\leq |T_0| < \sqrt{\frac{(24-1)22}{2}} <16.$ Since $T_0^2=23S_0^2-22,$ and since $23\cdot2^2-22=70$ and $23\cdot3^2-22=185$ are not squares, we have $S_0=1$, $T_0=\pm 1$. Therefore, there are at most two classes, represented by $1+\sqrt{23}$ and  $-1+\sqrt{23}$. Because $(1+\sqrt{23})(1-\sqrt{23})=-22$ is exactly the right-hand side of \eqref{eq:pell}, Nagell's criterion for two solutions to be associated reduces in the present case to $\frac{-1+\sqrt{23}}{1+\sqrt{23}}\in\Z[\sqrt{23}]$. Since $\frac{-1+\sqrt{23}}{1+\sqrt{23}}=\frac{12-\sqrt{23}}{11}\notin\Z[\sqrt{23}],$ there are exactly two classes.

By \cite[Theorem~109]{nagel}, every solution in either class is obtained by
multiplying its fundamental representative by a solution of $X^2-23Y^2=1.$ Hence, all such multipliers are $\pm u^k$, $k\in\Z$. In the first class, the solutions with both coefficients positive are \[ (1+\sqrt{23})u^k,\qquad k\geq0. \] In the second class, the representative $-1+\sqrt{23}$ has negative rational coefficient, while $(-1+\sqrt{23})u=91+19\sqrt{23}.$ Consequently, the solutions with both coefficients positive in this class are \[ (91+19\sqrt{23})u^k,\qquad k\geq0. \] Within each of the two distinct classes the exponent is unique because $u$ has infinite order. Thus every positive solution occurs uniquely in exactly one of the two displayed families.\end{proof}

For our construction it is enough to use the first Pell orbit. Hence, for every integer $k\geq0$, define $T_k,S_k\in\Z$ by \begin{equation}\label{eq:pell-orbit} T_k+S_k\sqrt{23}=(1+\sqrt{23})u^k. \end{equation} Taking norms gives \begin{equation}\label{eq:norm} T_k^2-23S_k^2=-22. \end{equation} Multiplying the equation \eqref{eq:pell-orbit}  by $u$, for $k\geq 0$, gives the recurrence \begin{equation}\label{eq:recurrence} T_{k+1}=24T_k+115S_k,\qquad S_{k+1}=5T_k+24S_k,\qquad T_0=S_0=1. \end{equation} Throughout, congruences in $\Z[\sqrt{23}]$ are
understood coefficientwise. Since $u\equiv-1\pmod5$, we have \begin{equation}\label{eq:S-mod5} S_k\equiv(-1)^k\pmod5. \end{equation} Motivated by $5c+1=S^2$ and by the condition
$c=575x+120$ from Proposition~\ref{prop:factorization}, we define \begin{equation}\label{eq:c-x-definitions} c_k=\frac{S_k^2-1}{5}\in\Z,\qquad x_k=\frac{c_k-120}{575} =\frac{S_k^2-601}{2875}\in\Q. \end{equation} Equations \eqref{eq:norm} and \eqref{eq:c-x-definitions} give \begin{equation}\label{eq:triple-identities} 5c_k+1=S_k^2,\qquad 115c_k+1=T_k^2. \end{equation} We now determine exactly which terms of this Pell orbit make $x_k$ integral.

\begin{lemma}\label{lem:exact-indices} For every integer $k\geq0$, the following conditions are equivalent: \begin{enumerate} \item $x_k\in\Z$; \item $S_k^2\equiv601\pmod{2875}$; \item $k\equiv15\quad\text{or}\quad440\pmod{575}.$ \end{enumerate} \end{lemma}

\begin{proof} The equivalence of the first two conditions follows immediately from \eqref{eq:c-x-definitions}. Since $2875=125\cdot23,$ the second condition is equivalent to \begin{equation}\label{eq:two-integrality-congruences} S_k^2\equiv101\pmod{125}, \qquad S_k^2\equiv3\pmod{23}. \end{equation}

We first consider the congruence modulo $125$. Put $w=5+\sqrt{23}$ and $z=5w.$ Then $u=24+5\sqrt{23}=-1+5(5+\sqrt{23})=-(1-z)$. Since $z=5w$, we have $z^j\equiv0\pmod{125}$ for every $j\geq3$.
Thus, from $u=-(1-z)$ and the binomial theorem, for every $k\geq0$,
\[
u^k\equiv
(-1)^k\big(
1-kz+\tbinom{k}{2} z^2
\big)
\pmod{125}.
\] Equivalently, \[ u^k\equiv (-1)^k\big( 1-5kw+25\tbinom{k}{2}w^2 \big) \pmod{125}. \] The coefficients of $\sqrt{23}$ in $1+\sqrt{23}$, $(1+\sqrt{23})w$ and $(1+\sqrt{23})w^2$ are $1$, $6$, and $58$, respectively. Therefore, \begin{align} S_k &\equiv (-1)^k \big( 1-30k+25\cdot58\tbinom{k}{2} \big)\pmod{125}\notag\\ &\equiv (-1)^k(1-5k+100k^2) \pmod{125}. \label{eq:S-mod125} \end{align} Squaring gives \begin{equation}\label{eq:S2-mod125} S_k^2\equiv1-10k+100k^2\pmod{125}. \end{equation} Hence, $S_k^2\equiv101\pmod{125}$ if and only if  $100k^2-10k-100\equiv0\pmod{125}. $ Dividing by $5$ gives $20k^2-2k-20\equiv0\pmod{25}, $ which is equivalent to $ 10k^2-k-10\equiv0\pmod{25}. $ Hence,  $k=5q$ for some $q\in\Z$, and substitution gives $q\equiv3\pmod5.$ Since every step above is reversible, we get \begin{equation}\label{eq:k-mod25} S_k^2\equiv101\pmod{125} \quad\Longleftrightarrow\quad k\equiv15\pmod{25}. \end{equation}

We next consider congruences modulo $23$. Since $(\sqrt{23})^2=23\equiv0\pmod{23}$ and $u\equiv1+5\sqrt{23}\pmod{23},$ the binomial theorem, for every $k\geq 0$, gives \[ u^k\equiv1+5k\sqrt{23}\pmod{23}. \] Thus $(1+\sqrt{23})u^k\equiv1+(5k+1)\sqrt{23}\pmod{23},$ and hence \[ S_k\equiv5k+1\pmod{23}. \] The two square roots of $3$ modulo $23$ are $\pm7$, and every step above is reversible, so \begin{equation}\label{eq:k-mod23} S_k^2\equiv3\pmod{23} \quad\Longleftrightarrow\quad k\equiv15\quad\text{or}\quad3\pmod{23}. \end{equation}

Combining \eqref{eq:k-mod25} and \eqref{eq:k-mod23}, Chinese remainder theorem gives \[ k\equiv15\pmod{575} \] in the first case and \[ k\equiv440\pmod{575} \] in the second. This proves the lemma. \end{proof}

\begin{remark}\label{rem:second-orbit}
The same computation applied to the second orbit \eqref{eq:second-pell-orbit} shows that its terms give an integral value of $x$ for $k\equiv134$ or $559\pmod{575}$. The corresponding value of $c$ for $k=134$ already has $453$ decimal digits. Hence, among all Diophantine extensions $c$ of $\{5,115\}$ satisfying \eqref{eq:compatibility}, the smallest one is the value $c_{15}$ obtained below.
\end{remark}

\section{The counterexample family}

We now state the main result. \begin{theorem}\label{thm:main} Let $T_k,S_k\in\Z$ be defined by \eqref{eq:pell-orbit} for $k\geq 0$, and put \[ \mathcal A= \{k\in\Z_{\geq0}:k\equiv15\ \text{or}\ 440\pmod{575}\}. \] For every $k\in\mathcal A$, define $c_k,x_k$ by \eqref{eq:c-x-definitions} and put \begin{equation}\label{eq:y-definition} y_k=c_kx_k+1. \end{equation} Then the following statements hold. \begin{enumerate} \item $c_k,x_k,y_k$ are positive integers and $\{5,115,c_k\}$ is a Diophantine triple. \item $(x_k,\pm y_k)$ is an integral point on \begin{equation}\label{eq:family-curve} E_k:\qquad y^2=(5x+1)(115x+1)(c_kx+1). \end{equation} \item Each of $5x_k+1$, $115x_k+1$, and $c_kx_k+1$ is a nonsquare in $\Z$. \end{enumerate} Moreover, the triples and the corresponding counterexamples are pairwise distinct, and the least value of $c_k$ in this family is attained at $k=15$. \end{theorem}

\begin{proof} Lemma~\ref{lem:exact-indices} gives $x_k\in\Z$ for every $k\in\mathcal A$. The recurrence \eqref{eq:recurrence} shows that $T_k,S_k$ are positive for every $k\geq0$. Moreover, $S_{k+1}=5T_k+24S_k>S_k,$ so $(S_k)_{k\geq0}$ is strictly increasing. The least member of $\mathcal A$ is $15$. Hence, $S_k\geq S_{15}>S_1=29$ for every $k\in\mathcal A$. By \eqref{eq:c-x-definitions}, $c_k>168$ and then $x_k>0$, while $y_k>0$ follows immediately from \eqref{eq:y-definition}. The strict increase of $S_k$ also shows that the values $c_k$ are pairwise distinct. The identities in \eqref{eq:triple-identities}, together with $ 5\cdot115+1=24^2, $ show that $\{5,115,c_k\}$ is a Diophantine triple.

Also, \eqref{eq:c-x-definitions} gives \begin{equation}\label{eq:c-linear} c_k=575x_k+120=5\cdot115x_k+5+115. \end{equation} Therefore, \begin{equation}\label{eq:main-factorization} c_kx_k+1=(5x_k+1)(115x_k+1) \end{equation} and Proposition~\ref{prop:factorization} yields that $(x_k,\pm y_k)$ is an integral point on the elliptic curve \eqref{eq:family-curve}.

It remains to prove that the three factors are nonsquares. Every $k\in\mathcal A$ is divisible by $5$. A direct calculation gives $u^5\equiv-1\pmod{11}$. Hence, $S_k\equiv\pm1\pmod{11}, $ and therefore $ c_k\equiv0\pmod{11}.$ Reducing \eqref{eq:c-linear} modulo $11$ gives $0\equiv3x_k+10\pmod{11},$ so $x_k\equiv4\pmod{11}.$ It follows that \begin{equation}\label{eq:first-two-nonsquares} 5x_k+1\equiv115x_k+1\equiv10\pmod{11}. \end{equation} Since $10$ is a quadratic nonresidue modulo $11$, $5x_k+1$ and $115x_k+1$ are nonsquares in $\Z$.

For the third factor we use the prime $p=2851$, for which $u$ has order $50$ modulo $p$. Repeated multiplication gives $u^5\equiv2385+1778\sqrt{23}\pmod{2851},$ $u^{10}\equiv959+2186\sqrt{23}\pmod{2851},$ $u^{20}\equiv466+1778\sqrt{23}\pmod{2851},$ $u^{25}\equiv-1\pmod{2851}$ and \begin{equation}\label{eq:mod2851-seed} (1+\sqrt{23})u^{15}\equiv852+1227\sqrt{23}\pmod{2851}. \end{equation} By Lemma~\ref{lem:exact-indices}, every $k\in\mathcal A$ satisfies $k\equiv15\pmod{25}.$ Hence, \eqref{eq:mod2851-seed} and $u^{25}\equiv-1\pmod{2851}$ imply $S_k^2\equiv1227^2\equiv201\pmod{2851}.$ Using $5^{-1}\equiv2281\pmod{2851},$  $2875^{-1}\equiv594\pmod{2851},$ and \eqref{eq:c-x-definitions}, we obtain \begin{equation}\label{eq:third-certificate} c_k\equiv40\pmod{2851},\qquad x_k\equiv1884\pmod{2851},\qquad c_kx_k+1\equiv1235\pmod{2851}. \end{equation} Finally, $1235=5\cdot13\cdot19$ and $2851\equiv1\pmod5$, $2851\equiv4\pmod{13}$, $2851\equiv1\pmod{19}$. Since $19\equiv2851\equiv3\pmod4$, quadratic reciprocity gives $\legendre{1235}{2851} = \legendre{5}{2851} \legendre{13}{2851} \legendre{19}{2851} = 1\cdot1\cdot(-1) =-1.$ Thus $1235$ is a quadratic nonresidue modulo $2851$, so $c_kx_k+1$ is also a nonsquare. \end{proof}

\begin{corollary}\label{cor:conjecture}
For every $k\in\mathcal A$ we have $x_k\notin\{-1,0,d_-,d_+\}$. Consequently, Problem~4.8 in \cite{DujellaOpen} has a negative answer and Conjecture~1.4.5 in \cite{DujellaBook} is false.
\end{corollary}

\begin{proof}
We have $5\cdot115+1=24^2$ and \eqref{eq:triple-identities}. Using the standard identities for the regular extensions \eqref{d+-},
$ad_\pm+1=(at\pm rs)^2,$ $bd_\pm+1=(bs\pm rt)^2,$ $cd_\pm+1=(cr\pm st)^2,$ where $r^2=ab+1$, $s^2=ac+1$, and $t^2=bc+1$, for the Diophantine  triple $\{5,115,c_k\}$ we obtain
\[
 5d_\pm+1=(5T_k\pm24S_k)^2,\qquad 115d_\pm+1=(115S_k\pm24T_k)^2,\qquad c_kd_\pm+1=(24c_k\pm S_kT_k)^2.
\]
Hence, all three factors in \eqref{eq:family-curve} are squares at $x=d_\pm$, and all three are squares at $x=0$. By Theorem~\ref{thm:main}, none of them is a square at $x=x_k$, and $x_k>0$. Hence $x_k\notin\{-1,0,d_-,d_+\}$, and Remark~\ref{rem:criterion} applies. Therefore, Conjecture~1.4.5 does not hold.
\end{proof}

\section{The least member of the family}

By \eqref{eq:recurrence}, $T_{15}=47\,629\,391\,928\,872\,028\,421\,837\,349$ and
 $S_{15}=9\,931\,414\,749\,943\,516\,928\,215\,699.$ Consequently,
\begin{equation}\label{eq:first-cx}
\begin{aligned}
 c_{15}={}&19\,726\,599\,787\,079\,129\,775\,103\,071\,759\,543\,738\,099\,950\,374\,011\,720,\\
 x_{15}={}&34\,307\,130\,064\,485\,443\,087\,135\,776\,973\,119\,544\,521\,652\,824\,368.
\end{aligned}
\end{equation}
Thus $\{5,115,c_{15}\}$ is a Diophantine triple and
\[
 \bigl(x_{15},\pm(c_{15}x_{15}+1)\bigr)
\]
is an integral point on the associated elliptic curve.  By \eqref{eq:first-two-nonsquares} and \eqref{eq:third-certificate}, none of the three factors is a square.

By \eqref{eq:triple-identities} and $5\cdot115+1=24^2$, we have $\sqrt{(ab+1)(ac+1)(bc+1)}=24S_{15}T_{15}$, so the regular extension values of this triple are $d_\pm=120+1151c_{15}\pm48S_{15}T_{15}$, that is,
\begin{equation}\label{eq:d-values}
\begin{aligned}
 d_-={}&8\,569\,332\,541\,111\,894\,343\,538\,801\,554\,602\,513\,003\,703\,076\,192,\\
 d_+={}&45\,410\,624\,140\,523\,615\,630\,392\,927\,651\,668\,130\,503\,572\,757\,271\,903\,488.
\end{aligned}
\end{equation}
In particular, by \eqref{eq:first-cx} and \eqref{eq:d-values}, $0<d_-<x_{15}<d_+$.

By Remark~\ref{rem:second-orbit}, this is the least member of the family constructed in Theorem~\ref{thm:main} and, more generally, the least one arising from the pair $\{5,115\}$. It is also the explicit counterexample recorded in
Dujella's Errata and addenda \cite{DujellaErrata} and in the updated entry for
Problem~4.8 \cite{DujellaOpen}.

\section*{Acknowledgments}
The first author was supported by the Croatian Science Foundation Grant No.~IP-2022-10-5008 and by the European Union -- NextGenerationEU, project number uniri-iz-25-62-ALGEBRA. The authors thank Andrej Dujella for his valuable advice and support, and for recording the construction in the
companion problem list and on the errata and addenda page for his monograph.

\section*{Declaration on the use of generative AI}

OpenAI's GPT-5.6 Sol model, used in Pro mode through ChatGPT, assisted during
the discovery of the construction.  OpenAI Codex assisted with proof checking
and language editing. The generated candidate parameter family was formally proved by the authors using standard congruence methods and Pell sequence recurrences. The authors take full responsibility for the content.

\end{document}